\documentclass[12]{amsart}

\usepackage{amscd,amsxtra,amsthm}
\usepackage[all]{xy}
\usepackage{etex}
\usepackage{pictex}
\usepackage{graphicx}
\usepackage{mathtools}

\usepackage{comment}

\newtheorem{theorem}{Theorem}%[section]

\theoremstyle{definition}

\theoremstyle{corollary}
\newtheorem{corollary}[theorem]{Corollary}

\theoremstyle{conjecture}

\theoremstyle{remark}

\theoremstyle{proposition}

\usepackage[english]{babel}
\usepackage[utf8]{inputenc}
\usepackage{amsmath}
\usepackage{graphicx}
\usepackage[colorinlistoftodos]{todonotes}


\title{A short note on the order of the Zhang-Liu matrices over arbitrary fields}

\author[L. Betthauser]{Leo Betthauser}
\address{Department of Mathematics \\
PO Box 118105 \\
University of Florida \\
Gainesville, FL 32611-8105}
\email{lbetthauser@ufl.edu}

\author[J. Hiller]{Josh Hiller}
\address{Department of Mathematics \\
PO Box 118105 \\
University of Florida \\
Gainesville, FL 32611-8105}
\email{jphiller1@ufl.edu}

\date{\today}

\begin{document}
\maketitle

\begin{abstract}
We give necessary and sufficient conditions for the Zhang-Liu matrices to be diagonalizable over arbitrary fields and provide the eigen-decomposition when it is possible. We use this result to calculate the order of these matrices over any arbitrary field. This generalizes a result of the second author.
\end{abstract}

\section{Introduction}
 
The various generalizations of the pascal matrices have been an active subject of research since at least the early 1990's (the interested reader can see \cite{CandV},  \cite{Z}, \cite{ZandW} for background on some of the more well studied variations). The study of the order of these matrices over finite rings, however, is much more recent and can be traced to Deveci and Karaduman \cite{OandD}. In their article,  an explicit function which calculates the order of the generalized Pascal matrix of the first kind over the ring $\mathbb{Z}/n\mathbb{Z}$ was given. 

These matrices are defined in the following way:

$$ P_1(y)= (p)_{ij}  =\left\{ \begin{array}{cl} y^{j-i}\left(\begin{array}{c} j-1 \\ i-1\end{array}\right) & \mbox{ if } j \geq i\\ & \\  0 & \mbox{ otherwise.} \end{array} \right.$$

Where $y\in \mathbb{Z}/n\mathbb{Z}.$ However, analogous results for the symmetric Pascal matrices and the generalized Pascal matrices of the second kind over finite fields have remained elusive (despite the partial results in \cite{ZL}, \cite{OandD}, \cite{Hiller} and \cite{L}).  The Pascal matrices of the second kind are the  square matrices defined as follows: 

$$ P_2(x)= (p_2)_{ij}  =\left\{ \begin{array}{cl} x^{j+i-2}\left(\begin{array}{c} j-1 \\ i-1\end{array}\right) & \mbox{ if } j \geq i\\ & \\  0 & \mbox{ otherwise.} \end{array} \right.$$

where $x \in \mathbb{F}^{\times}$.

One type of matrix which has yet to be examined in this context are the Zhang-Liu matrices. These were introduced in 1998 in \cite{ZL}. The utility of these matrices is that, in essence, they ``give" us both generalizations of the Pascal matrix. They are defined as:

$$ Q(y,x)= (\rho)_{ij}  =\left\{ \begin{array}{cl} (y^{j-i})(x^{j+i-2})\left(\begin{array}{c} j-1 \\ i-1\end{array}\right) & \mbox{ if } j \geq i\\ & \\  0 & \mbox{ otherwise.} \end{array} \right.$$

\noindent where $x$ is in the multiplicative group of a field, and $y$ is any element of the field. 

In this note, we give necessary and sufficient conditions for $Q(y,x)$ to be diagonalizable over a given field. We also determine the exact order of $Q(y,x)$ as a function of $y$ and $x$.

Throughout this note, $p$ will denote a prime,  we will assume that $y\in \mathbb{F}$ is a field with multiplicative group $\mathbb{F}^{\times}$ and that $x\in \mathbb{F}^{\times}.$ We will also assume that all matrices are square of dimension $n\geq 2$ and we denote $|a|$ as  the multiplicative order of an element $a$ of a field. Also,  we operate under the convention of \cite{Z} and \cite{CandV} that $0^0=1.$ 

\section{Results}

The desired theorems will stem from the following factorization.

 \begin{theorem} If $x \not\in\{1,-1\}$and $z=\frac{yx}{x^2-1}$ then the following equation holds: 

$$Q(y,x)=P_1(z)D(x^2)P_1(-z),$$

\noindent where $D(\alpha)$ is the diagonal matrix consisting of nonzero entries $(d)_{i,i}=(\alpha)^{i-1}$ where $\alpha\in \mathbb{F}$.
\end{theorem}

\begin{proof}

Given that $x\not\in \{1,-1\},$ then $x^2-1 \not= 0$ and so $\frac{yx}{x^2-1}$ exists. Since all the matrices in the product are upper triangular, we will concern ourselves only with the case when  $j \geq i$. We denote the $i,j$ entry of $P_1(z)$ by $p_{i,j}$ and the $i,j$ entry of $P_1(-z)$ by $p'_{i,j}.$ Then, letting $b_{i,j}$ be the $i,j$ entry of the matrix $P(z)D(x^2)P(-z)$ we see that by defining   

$$a_{i,j}:=\sum_{k=1}^{n}{d_{i,k}p'_{k,j}}=(x^2)^{i-1}\left(\frac{-xy}{x^2-1}\right)^{j-i}\binom{j-1}{i-1}$$

\noindent we can  conclude that

$$b_{i,j}=\sum_{k=1}^{n}{p_{i,k}a_{k,j}}=\sum_{k=i}^{j}{\left(\frac{xy}{x^2-1}\right)^{k-i}\binom{k-1}{i-1}(x^2)^{k-1}\left(\frac{-xy}{x^2-1}\right)^{j-k}\binom{j-1}{k-1}}$$

$$=\sum_{k=i}^{j}{(-1)^{j-k}\left(\frac{xy}{x^2-1}\right)^{j-i}\binom{j-1}{k-1}\binom{k-1}{i-1}(x^2)^{k-1}}$$

$$=\left(\frac{xy}{x^2-1}\right)^{j-i}\sum_{k=i}^{j}{(-1)^{j-k}\binom{j-1}{i-1}\binom{j-i}{k-i}(x^2)^{k-1}}$$

$$=\left(\frac{xy}{x^2-1}\right)^{j-i}\binom{j-1}{i-1}\sum_{k=i}^{j}{(-1)^{j-k}\binom{j-i}{k-i}(x^2)^{k-1}}$$

$$=\left(\frac{xy}{x^2-1}\right)^{j-i}\binom{j-1}{i-1}\sum_{k=0}^{j-i}{\binom{j-i}{k}(-1)^{j-i-k}(x^2)^{k+i-1}}$$

$$=\left(\frac{xy}{x^2-1}\right)^{j-i}\binom{j-1}{i-1}(x^2)^{i-1}\sum_{k=0}^{j-i}{\binom{j-i}{k}(-1)^{j-i-k}(x^2)^{k}}$$

$$=\left(\frac{xy}{x^2-1}\right)^{j-i}\binom{j-1}{i-1}(x^2)^{i-1}(x^2-1)^{j-i}$$

$$= (xy)^{j-i}\binom{j-1}{i-1}x^{j+i-2}.$$

$$=(y^{j-i})(x^{j+i-2})\binom{j-1}{i-1}$$

\noindent Proving that $b_{i,j}=\rho_{i,j}.$
\end{proof}

This factorization leads to the following observation regarding the eigenvectors of $Q(y,x)$.

\begin{corollary}
If $x \not\in\{1,-1\}$, then the eigenvectors of $Q(y,x)$ are the columns of $P(\frac{yx}{x^2-1}).$ 
\end{corollary}

\begin{proof}
It is simple to verify that $P_1(z)^{-1}=P_1(-z)$ (in the case of the real numbers, this has been done explicitly in \cite{CandV}, and in fact, their proof holds for arbitrary commutative rings with identity).  Thus the factorization provided by Theorem 6 is the eigen-decomposition of $Q(y,x)$. 
\end{proof}

Combining this, with the observation that if $x\in\{-1,1\}$ and $y\not= 0$ that $Q(y,x)$ is deficient leads to the following desired theorem.

\begin{theorem}
The Zhang-Liu matrix $Q(y,x)$ is diagonalizable if and only if $x\not\in \{-1,1\}$ or $y=0$.
\end{theorem}

The next corollary  gives us the order of $Q(y,x)$. 

\begin{corollary}
Let $k= |x^2| $,

\begin{enumerate}

\item If $x\not\in \{1,-1\}$ then $|Q(x)|$ is $k$.

\item If $x\in \{1,-1\}$ and the charecteristic of $F$ is $q \not= 0$ then $|Q(x)| = q$.

\item If $x\in \{1,-1\}$ and the charecteristic of $F$ is $0$ then $|Q(x)| = \infty.$

\end{enumerate}

\end{corollary}

\begin{proof}
If $x\not\in \{1,-1\}$ then by Theorem 1, $Q(y,x)$ is diagonalizable over $\mathbb{F}$ with all eigenvalues a power of $x^2$ thus $|Q(y,x)|=|x^2|.$ 

To complete proof, note that $Q(y,1)=P(y)$, thus $Q(1,1)=P(1)$ which by  \cite{OandD} is of order $p$ if $\mathbb{F}$ is of characteristic $p$. A similar argument holds if $x=-1$. 
\end{proof}

We remark that this last corollary is a vast generalization of Theorem 2.4 of \cite{Hiller} which says that if  $F$ is a field with $p$ elements and $ |P_2(x))| < p,$ then $|P_2 (x)|   = |x^2| $.

\section*{Acknowledgement}
\noindent The authors would like to thank Ali Unca for his helpful insights and discussion.

%\section{Acknowledgements}
%The authors would like to thank Ali Unka for his careful reading and helpful discussion. 

%/ Xxxxxxxxxxxxxxx BIBTEK

\bibliographystyle{amsplain}

%#################
%###################
\begin{comment}

% Template for Elsevier CRC journal article
% version 1.1 dated 16 March 2010

% This file (c) 2010 Elsevier Ltd.  Modifications may be freely made,
% provided the edited file is saved under a different name

% This file contains modifications for Procedia Computer Science
% but may easily be adapted to other journals

% Changes since version 1.0
% - elsarticle class option changed from 1p to 3p (to better reflect CRC layout)

%-----------------------------------------------------------------------------------

%% This template uses the elsarticle.cls document class and the extension package ecrc.sty
%% For full documentation on usage of elsarticle.cls, consult the documentation "elsdoc.pdf"
%% Further resources available at http://www.elsevier.com/latex

%-----------------------------------------------------------------------------------

%%%%%%%%%%%%%%%%%%%%%%%%%%%%%%%%%%%%%%%%%%%%%%
%%%%%%%%%%%%%%%%%%%%%%%%%%%%%%%%%%%%%%%%%%%%%%
%%                                          %%
%% Important note on usage                  %%
%% -----------------------                  %%
%% This file must be compiled with PDFLaTeX %%
%% Using standard LaTeX will not work!      %%
%%                                          %%
%%%%%%%%%%%%%%%%%%%%%%%%%%%%%%%%%%%%%%%%%%%%%%
%%%%%%%%%%%%%%%%%%%%%%%%%%%%%%%%%%%%%%%%%%%%%%

%% The '3p' and 'times' class options of elsarticle are used for Elsevier CRC
\documentclass[3p,times]{elsarticle}

%% The `ecrc' package must be called to make the CRC functionality available
\usepackage{ecrc}

%% The ecrc package defines commands needed for running heads and logos.
%% For running heads, you can set the journal name, the volume, the starting page and the authors

%% set the volume if you know. Otherwise `00'
\volume{00}

%% set the starting page if not 1
\firstpage{1}

%% Give the name of the journal
\journalname{Procedia Computer Science}

%% Give the author list to appear in the running head
%% Example \runauth{C.V. Radhakrishnan et al.}
\runauth{}

%% The choice of journal logo is determined by the \jid and \jnltitlelogo commands.
%% A user-supplied logo with the name <\jid>logo.pdf will be inserted if present.
%% e.g. if \jid{yspmi} the system will look for a file yspmilogo.pdf
%% Otherwise the content of \jnltitlelogo will be set between horizontal lines as a default logo

%% Give the abbreviation of the Journal.
\jid{procs}

%% Give a short journal name for the dummy logo (if needed)
\jnltitlelogo{Procedia Computer Science}

%% Hereafter the template follows `elsarticle'.
%% For more details see the existing template files elsarticle-template-harv.tex and elsarticle-template-num.tex.

%% Elsevier CRC generally uses a numbered reference style
%% For this, the conventions of elsarticle-template-num.tex should be followed (included below)
%% If using BibTeX, use the style file elsarticle-num.bst

%% End of ecrc-specific commands
%%%%%%%%%%%%%%%%%%%%%%%%%%%%%%%%%%%%%%%%%%%%%%%%%%%%%%%%%%%%%%%%%%%%%%%%%%

%% The amssymb package provides various useful mathematical symbols
\usepackage{amssymb}
%% The amsthm package provides extended theorem environments
%% \usepackage{amsthm}

%% The lineno packages adds line numbers. Start line numbering with
%% \begin{linenumbers}, end it with \end{linenumbers}. Or switch it on
%% for the whole article with \linenumbers after \end{frontmatter}.
%% \usepackage{lineno}

%% natbib.sty is loaded by default. However, natbib options can be
%% provided with \biboptions{...} command. Following options are
%% valid:

%%   round  -  round parentheses are used (default)
%%   square -  square brackets are used   [option]
%%   curly  -  curly braces are used      {option}
%%   angle  -  angle brackets are used    <option>
%%   semicolon  -  multiple citations separated by semi-colon
%%   colon  - same as semicolon, an earlier confusion
%%   comma  -  separated by comma
%%   numbers-  selects numerical citations
%%   super  -  numerical citations as superscripts
%%   sort   -  sorts multiple citations according to order in ref. list
%%   sort&compress   -  like sort, but also compresses numerical citations
%%   compress - compresses without sorting
%%
%% \biboptions{comma,round}

% \biboptions{}

% if you have landscape tables
\usepackage[figuresright]{rotating}

% put your own definitions here:
%   \newcommand{\cZ}{\cal{Z}}
%   \newtheorem{def}{Definition}[section]
%   ...

% add words to TeX's hyphenation exception list
%\hyphenation{author another created financial paper re-commend-ed Post-Script}

% declarations for front matter

\begin{document}

\begin{frontmatter}

%% Title, authors and addresses

%% use the tnoteref command within \title for footnotes;
%% use the tnotetext command for the associated footnote;
%% use the fnref command within \author or \address for footnotes;
%% use the fntext command for the associated footnote;
%% use the corref command within \author for corresponding author footnotes;
%% use the cortext command for the associated footnote;
%% use the ead command for the email address,
%% and the form \ead[url] for the home page:
%%
%% \title{Title\tnoteref{label1}}
%% \tnotetext[label1]{}
%% \author{Name\corref{cor1}\fnref{label2}}
%% \ead{email address}
%% \ead[url]{home page}
%% \fntext[label2]{}
%% \cortext[cor1]{}
%% \address{Address\fnref{label3}}
%% \fntext[label3]{}

\dochead{}
%% Use \dochead if there is an article header, e.g. \dochead{Short communication}

\title{}

%% use optional labels to link authors explicitly to addresses:
%% \author[label1,label2]{<author name>}
%% \address[label1]{<address>}
%% \address[label2]{<address>}

\author{}

\address{}

\begin{abstract}
%% Text of abstract
\end{abstract}

\begin{keyword}
%% keywords here, in the form: keyword \sep keyword

%% MSC codes here, in the form: \MSC code \sep code
%% or \MSC[2008] code \sep code (2000 is the default)

\end{keyword}

\end{frontmatter}

%%
%% Start line numbering here if you want
%%
% \linenumbers

%% main text
\section{}
\label{}

%% The Appendices part is started with the command \appendix;
%% appendix sections are then done as normal sections
%% \appendix

%% \section{}
%% \label{}

%% References
%%
%% Following citation commands can be used in the body text:
%% Usage of \cite is as follows:
%%   \cite{key}         ==>>  [#]
%%   \cite[chap. 2]{key} ==>> [#, chap. 2]
%%

%% References with BibTeX database:

\bibliographystyle{elsarticle-num}
\bibliography{<your-bib-database>}

%% Authors are advised to use a BibTeX database file for their reference list.
%% The provided style file elsarticle-num.bst formats references in the required Procedia style

%% For references without a BibTeX database:

% \begin{thebibliography}{00}

%% \bibitem must have the following form:
%%   \bibitem{key}...
%%

% \bibitem{}

% \end{thebibliography}

\end{document}